\documentclass[11pt]{article}
\usepackage{amsmath,amssymb,amsthm, epsfig}
\usepackage{amsfonts}
\usepackage{amsmath}
\usepackage{amssymb}
\usepackage{color}
\usepackage{palatino}
\usepackage[mathscr]{eucal}

\title{\bf Hessian continuity at degenerate points \\ in nonvariational elliptic problems}

\author{ \textsc{Eduardo V. Teixeira} \\ \textit{\footnotesize Universidade Federal do Cear\'a}  \\ \textit{\footnotesize Fortaleza, CE, Brazil} }

\date{}

\def \suchthat {\ \big | \ }

\def \L {\mathfrak{L}}

\newtheorem{theorem}{Theorem}
\newtheorem{lemma}[theorem]{Lemma}

\theoremstyle{definition}

\theoremstyle{remark}

\numberwithin{equation}{section}

\newcommand{\intav}[1]{\mathchoice {\mathop{\vrule width 6pt height 3 pt depth  -2.5pt
\kern -8pt \intop}\nolimits_{\kern -6pt#1}} {\mathop{\vrule width
5pt height 3  pt depth -2.6pt \kern -6pt \intop}\nolimits_{#1}}
{\mathop{\vrule width 5pt height 3 pt depth -2.6pt \kern -6pt
\intop}\nolimits_{#1}} {\mathop{\vrule width 5pt height 3 pt depth
-2.6pt \kern -6pt \intop}\nolimits_{#1}}}

\begin{document}
\maketitle

\begin{abstract}
Established in the 30's, Schauder {\it a priori} estimates are among the most classical and powerful tools in the analysis of problems ruled by 2nd order elliptic PDEs. Since then, a central problem in regularity theory has been to understand Schauder type estimates fashioning particular borderline scenarios. In such context, it has been a common accepted aphorism that the continuity of the Hessian of a solution could never be superior than the continuity of the medium. Notwithstanding, in this article we show that solutions to uniformly elliptic, linear equations with $C^{0,\bar{\epsilon}}$ coefficients are of class $C^{2,\alpha}$, for any $0 < \bar{\epsilon} \ll \alpha < 1$, at Hessian degenerate points, $\mathscr{H}(u):=\{X \suchthat D^2u(X) = 0\}$. In fact we develop a more general regularity result at such Hessian degenerate points, featuring into the theory of fully nonlinear equations. Insofar as the optimal modulus of continuity  for the Hessian is concerned, the result of this paper is the first one in the literature to surpass the inborn obstruction from the sharp Schauder {\it a priori} regularity theory.



\medskip


{
}

\end{abstract}

\section{Introduction}

Among the finest treasures of the theory of elliptic PDEs, Schauder {\it a priori} regularity estimates assure that solutions to a linear, uniformly elliptic equation with $C^{0,\theta}$ data, $0 < \theta < 1$, i.e., functions $u$ satisfying 
\begin{equation}\label{linear eq}
	\L u := a_{ij}(X) D_{ij} u = f(X)
\end{equation}
where 
$$
	0< \lambda \le a_{ij}(X) \le \Lambda, \quad a_{ij}, f \in C^{0,\theta},
$$
are locally of class $C^{2,\theta}$.  Furthermore, there exists a constant $C>0$, depending only upon dimension, ellipticity constants $(\lambda, \Lambda)$, and the $\theta$--H\"older continuity of the data, $\|a_{ij}\|_{C^{0,\theta}}$ and $\|f\|_{C^{0,\theta}}$, such that
\begin{equation}\label{Schauder est}
	\|u\|_{C^{2,\theta}(B_{1/2})} \le C \cdot \|u\|_{L^\infty(B_1)}.
\end{equation}

\par

For comprehensives reference on such a theory, we cite the classical books \cite[Chapter 6]{GT} and also \cite[Chapter 6]{M} . It is also interesting to read \cite{Simon} and references therein. 

\par

The importance of such an estimate to the theory of PDEs and its vast range of applications would hardly be exaggerated. 
Proven in the 30's by the Polish mathematician, Juliusz Schauder,  \cite{S1},  Schauder {\it a priori} estimate \eqref{Schauder est} is sharp in several ways. Estimate \eqref{Schauder est} does not hold true in the case $\theta =0$, i.e.  solutions to elliptic equations \eqref{linear eq} with merely continuous data are not necessarily of class $C^2$. Not even $C^{1,1}$ estimates are in general available equations with continuous sources. Similarly, Schauder {\it a priori} estimate also breaks down at the upper endpoint, $\theta =1$. That is, solutions to elliptic equations \eqref{linear eq}
with Lipchitz data {\it are not} necessarily $C^{2,1}$. Establishing optimal regularity estimates in borderline cases for particular problems involves, in general, new, deep and robust techniques. A classical example we bring up here is the theory of obstacle-type free boundary problems
\begin{equation}\label{OP}
	\L v \approx \chi_{\{v > 0 \}}.
\end{equation}
Classical elliptic regularity theory gives that a solution to \eqref{OP} is of class $C^{1,\alpha}$ for any $\alpha < 1$. Proving that a solution is indeed $C^{1,1}$ involves a deep and much finer analysis, for instance quasi-monotonicity formulae \cite{CJK, Sh}. 
 
\par

Another decisive sharpness aspect of the Schauder regularity theory concerns the best exponent for H\"older continuity of the Hessian of a solution to \eqref{linear eq}. That is, fixed $0< \alpha_0 < 1$, a solution to a uniformly elliptic equation with $C^{0,\alpha_0}$ data is of class $C^{2,\alpha_0}$, but it may fail to belong to $C^{2, \alpha_0 + \delta}$, $\delta > 0$. Up to our knowledge, there has been no significant advances on the issue of surpassing the universal $C^{2, \alpha_0}$ regularity obstruction, at least for certain analytical meaningful points. Of particular interest are the Hessian degenerate points of the solution:
$$
	\mathscr{H}(u) :=  \left ( D^2u \right )^{-1} \{0\}.
$$

\par

The above discussion brings us to the main result of this present work. We shall establish in this manuscript that solutions to a linear elliptic equation
$$
	a_{ij}(X) D_{ij} u =0, 
$$
with $a_{ij} \in C^{0, \bar{\epsilon}}$, $0 < \bar{\epsilon} \ll 1$, is of class $C^{2,1^{-}}$ at any point $Y \in \mathscr{H}(u)$. The symbol $C^{2,1^{-}}$ means $C^{2,\alpha}$ for any $\alpha<1$. This is an unexpected gain of smoothness at Hessian degenerate points, beyond the continuity of the media.  It is furthermore simple to see that if $Y \not \in \mathscr{H}(u)$ such a result cannot hold in general. Thus, our result is, in this perspective, optimal. The very same conclusion is obtained if one assumes only that $a_{ij}$ is Dini continuous. To the best of our knowledge, our result is the first one, in the context of non-divergence elliptic regularity theory, to outmatch the classical  Schauder {\it a priori} estimate, insofar as the continuity of the media is concerned. 

\par

To exemplify the gain of smoothness provided by the result above mentioned, suppose, for the sake of illustration, that $u$ solves a linear equation
$$
	a_{ij}(X) D_{ij} u = 0,
$$
where $a_{ij}$ is uniform elliptic and $a_{ij} \in C^{0,0.1}$. Classical Schauder regularity theory gives that $D^2 u \in C^{0,0.1}$ at any interior point. However, at a degenerate Hessian point $Y \in \mathscr{H}(u)$, in fact $D^2u$ is much smoother, for instance,  $D^2u \in C^{0, 0.999}$, i.e., $D^2u$ is asymptotically Lipschitz continuous. 


\par

The solution designed for the proof of such a result is, in its very nature, nonlinear. Thus, for the sake of completeness, we shall state and proof a more general result, in the context of fully nonlinear elliptic equations. 

\par

Hereafter $Q_1 \subset \mathbb{R}^d$ denotes the unit cube in the $d$-dimensional euclidean space and $\mathrm{Sym}(d)$ stands for the space of  $d\times d$ symmetric matrices. Throughout this paper we shall work under uniform ellipticity condition on the operator $F\colon Q_1 \times \mathrm{Sym}(d) \to \mathbb{R}$, i.e., we assume there exist two positive constants $0 < \lambda \le \Lambda$ such that, for any $M \in \mathrm{Sym}(d)$,  $X \in Q_1$,
\begin{equation}  \tag{H0}\label{H0}
    \lambda \|P\| \le F(X, M+P) - F(X, M) \le \Lambda \|P\|, \quad \forall P \ge 0.
\end{equation}

\par

As to access estimates on the Hessian of a solution, in this article we shall assume that the model, constant coefficient equation has {\it a priori} $C^{2,\alpha_F}$ local estimates, for some $0 <\alpha_F \le 1$. More precisely, we assume
\begin{equation} \tag{H1}\label{H1}
    F(0, D^2h) = 0, \text{ in } Q_1 \quad \text{implies} \quad \|h\|_{C^{2,\alpha_F}(Q_{1/2})} \le \Theta \cdot  \|h\|_{L^\infty(Q_1)},
\end{equation}
for some constant $\Theta >0$.

\par

Since we will deal with variable coefficient equations, in accordance to \cite{C1} (see also \cite{CC}), throughout this article we impose an $L^n$ type of 
$C^{0, \bar{\epsilon}}$ continuity of the coefficients. More precisely,  measuring the oscillation of the coefficients at $0\in B_1 \subset \mathbb{R}^d$, by
$$
     {\beta}_F(X) := \sup\limits_{N \in \mathrm{Sym}(n)}  \dfrac{\left |F(X,N) - F(0,N) \right |}{1+\|N\|},
$$
we shall impose the existence of a constant $C_1 > 0$ such that 
\begin{equation}\tag{H2}\label{H2}
    \intav{Q_r}   |  {\beta}_F(X)  |^n dX \le C_1^n \cdot  r^{n \bar{\epsilon}},
\end{equation}
for some $0< \bar{\epsilon} \ll 1$. 

\par

For notation convenience, we shall call
\begin{equation}\label{def Ln norm 1}
	[F]_{n,\bar{\epsilon}} := \inf \{C_1 \in [0, \infty) \suchthat \eqref{H2} \text{ hods} \}.
\end{equation}
Similarly, we measure the oscillation of the source function $f\colon Q_1 \to \mathbb{R}$ around $0$ by
\begin{equation}\label{def Ln norm 2}
	[f]_{n,\gamma}:= \inf \{K \in [0, \infty) \suchthat \intav{Q_r}   | f(X)  |^n dX \le K^n \cdot  r^{n \gamma} \}.
\end{equation}
When $[f]_{n,\gamma} < +\infty$, we say $f$ is $C^{0,\gamma}$ continuous at $0$ in the $L^n$ sense. 

\par

We recall that under $C^{0, \bar{\epsilon}}$ continuity of the coefficients (hypothesis \eqref{H2}) and $C^{2,\alpha_F}$ \textit{a priori} estimates, $0 < \bar{\epsilon} < \alpha_F$, for $F(0,D^2h) = 0$ (hypothesis \eqref{H1}), Luis Caffarelli in \cite{C1} shows that solutions to  the variable coefficients equation, $F(X, D^2\xi) = 0$ is $C^{2,\bar{\epsilon}}$, with appropriate \textit{a priori} estimates. In a pararel to the classical, linear Schauder regularity theory, the exponent $\bar{\epsilon}$ in Caffarelli's $C^{2}$ estimates cannot be surpassed, in general. 

\par

Before continuing, we set that throughout this paper, any given constant $\kappa$ that depends only upon dimension $d$, ellipticity $(\lambda, \Lambda)$, $[F]_{n,\bar{\epsilon}}$,  $[f]_{n,\gamma}$ and the $C^{2,\alpha_F}$ regularity estimates for the constant coefficient equation, i.e., the constant $\Theta$ in hypothesis \eqref{H1} will be called {\it universal}. We now state the general regularity Theorem at Hessian degenerate points of this present manuscript.

\par

\begin{theorem}\label{main} Let $u\in C(Q_1)$ be a viscosity solution to $F(X, D^2u) = f(X)$ in $Q_1$. Assume $F \colon Q_1 \times \mathrm{Sym}(d) \to \mathbb{R}$ satisfy conditions \eqref{H0}, \eqref{H1} and \eqref{H2} and that $0 \in \mathscr{H}(u)$. Assume further that $f$ is $C^{0,\gamma}$ continuous at $0$ in the $L^n$ sense, for some $ 0 < \gamma < 1$. Then $u$ is $C^{2, \min\{\alpha_F^{-}, \gamma\}}$ at the origin, i.e., for any $\beta \in (0,\alpha_F) \cap (0, \gamma]$, 
$$
	|u(X) - [u(0) + \nabla u(0) \cdot X]| \le C_\beta |X|^{2+\beta},
$$
for a constant $C_\beta >0$ that depends only on $\beta$, $|f(0)|$, and universal parameters. 

\end{theorem}

\par

We highlight once more that the key information provided in Theorem \ref{main} is that at Hessian degenerate points, solutions disregard the rough continuity of the media (coefficients), and its regularity theory is asymptotically as strong as the constant coefficient one. Within the classical regularity theory for 2nd order elliptic PDEs, this is at the very least a surprising result. 

\par

Notice, furthermore, that since harmonic functions are of class $C^{2,1}$, indeed it follows as a Corollary of Theorem \ref{main} that solutions to a linear, uniform elliptic  homogeneous equation
$$
	a_{ij}(X) D_{ij} u = 0,
$$
with $C^{0,\bar{\epsilon}}$ coefficients are of class $C^{2,1^{-}}$ at Hessian degenerate points. The very same $C^{2, 1^{-}}$ regularity at Hessian degenerate points holds true for fully nonlinear equation governed by an operator $F \in C^{1,\bar{\epsilon}}(Q_1, \times \mathrm{Sym}(d))$ that is convex or concave w.r.t. the matrix variable $M$.

\par

A careful analysis of the proof of Theorem \ref{main}, to be delivered in the next section, when projected to the linear setting, reveals that the $C^{2,1^{-}}$ regularity at Hessian degenerate points holds as soon as the coefficients are ``continuous enough" as to allow {\it a priori} $C^{2}$ estimates for $\L$-harmonic functions. A classical example of such condition is the so called Dini continuity of the coefficients. Recall a function $\phi$ is said to be Dini continuous at $0$ if 
$$
	\int_0^1 \frac{\sup_{B_t} |\phi(X) - \phi(0)|}{t} dt < \infty. 
$$  
In particular any $C^{0, {\epsilon}}$ H\"older continuous function  is  Dini continuous.  We state the conclusion of the above discussion as a Theorem.

\par

\begin{theorem}\label{main linear} Let $u\in C(Q_1)$ be a viscosity solution to $a_{ij}(X) D_{ij} u = 0$ in $Q_1$, where $a_{ij}$ is uniform elliptic and Dini continuous matrix. Then $u$ is $C^{2, 1^{-}}$ at any Hessian degenerate point, $Y \in \mathscr{H}(u) \cap B_{1/2}$. That is,  for any $\alpha < 1$ 
$$
	|u(X) - [u(Y) + \nabla u(Y) \cdot (X-Y)]| \le C_\alpha |X-Y|^{2+\alpha},
$$
for a constant $C_\alpha >0$ that depends only on $(1-\alpha)$,  dimension, ellipticity constants and Dini continuity of the coefficients. 

\end{theorem}

\par

We conclude this introduction by mentioning that the statement of Theorem \ref{main linear} as well as the insights for the proof of Theorem \ref{main} have their essential roots in the affluent theory of geometric free boundary problems, see \cite{CS} and also \cite[Chapter  5]{F}.   Of particular interest are the classes of {\it pseudo} free boundary problems of the form
$$
	\max \left \{ \L\phi - \phi^{\mu}, -\phi \right\} = 0, \quad \mu >0,
$$
for some nonhomogeneous elliptic operator $\L$. The limiting case $\mu = 0$ represents the classical obstacle problem, already mentioned above. For $\mu > 0$, solutions are of class $C^2$ across the free interface $\partial \{\phi > 0\}$. Thus, all free boundary points are Hessian degenerate. In accordance to Theorem \ref{main linear}, $\phi \in C^{2,1^{-}}$ at free boundary points, neglecting the eventual rough continuity of the coefficients of $\L$.

\section{Proof}


\subsection{Hessian degenerate approximations}
In this first part of the proof we establish a primary compactness result. It states that if the coefficients are nearly constant, the source is nearly zero and the Hessian of the solution at 0 is tiny enough, then one can find an $\mathscr{F}$-harmonic function close to $u$, for which $0$ is a Hessian degenerate point. 

\begin{lemma}\label{comp} Let $u\in C(Q_1)$ be a viscosity solution to $F(X, D^2u) = f(X)$ in $Q_1$ and $|u|\le 1$. Assume $F \colon Q_1 \times \mathrm{Sym}(d) \to \mathbb{R}$ satisfies conditions \eqref{H0}, \eqref{H1}, \eqref{H2} and $f(0) = 0$. Given $\delta > 0$, there exists an $\varepsilon$, depending only on $\delta$ and universal parameters, such that if 
$$
	[F]_{n,\bar{\epsilon}} + [f]_{n,\gamma} + |D^2u(0)| \le \varepsilon,
$$
then we can find a function $h\in C(Q_{1/2})$, satisfying
$$
	\mathscr{F}(D^2h) = 0 \text{ in } Q_{1/2} \quad \text{ and } \quad D^2h(0) = 0,
$$
where $\mathscr{F} \colon \mathrm{Sym}(d) \to \mathbb{R} $ is under the conditions \eqref{H0} and \eqref{H1} and
$$
	\sup\limits_{Q_{1/2}} |u - h| \le \delta.
$$
\end{lemma}

\begin{proof} Suppose, for the sake of contradiction, that the thesis of the Lemma fails. This means that there exist a sequence of functions 
$$
	u_j \in C(Q_1) \quad \text{ with} \quad |u_j| \le 1,
$$ 
a sequence of operators  
$$
	F_j \colon Q_1 \times \mathrm{Sym}(d) \to \mathbb{R} \quad \text{ satisfying conditions } \eqref{H0}, ~ \eqref{H1} \text{ and } \eqref{H2},
$$
and a sequence of sources $f_j$ with $f_j(0) = 0$, all linked by the equation
\begin{equation}\label{proof comp eq2}
	F_j(X, D^2u_j) = f_j(X),
\end{equation}
in the viscosity sense. There also holds
\begin{equation}\label{proof comp eq1}
	[F_j]_{n,\bar{\epsilon}} + [f_j]_{n,\gamma} + |D^2u_j(0)| = \text{o}(1),
\end{equation}
as $j \to \infty$. However, for a fixed $\delta_0>0$, 
\begin{equation}\label{proof comp eq3}
	\inf\limits_{h \in \mathbb{H}} \|u_j - h\|_{L^\infty(Q_{1/2})} \ge \delta_0,
\end{equation}
where
$$
\mathbb{H} := \left \{ 
	\begin{array}{c}
		h \suchthat  \mathscr{F}(D^2h) = 0 \text{ in } Q_{1/2}, \text{ for some }\mathscr{F} \text{ verifying } \eqref{H0} 
		\text{ and }  \eqref{H1}  \\
		\text{ such that } 0 \text{ belongs to } \mathscr{H}(h)
	\end{array}	
	\right \}.
$$

It follows from \eqref{proof comp eq1}, \eqref{proof comp eq2} and standard reasoning that, up to a subsequence, we can assume,
$$
	\begin{array}{ll}	
		F_j \to \mathscr{F}_\infty, & \text{ locally uniformly in } Q_1 \times \mathrm{Sym}(d) \\
		u_j \to u_\infty, & \text{ locally uniformly in } Q_1.
	\end{array}
$$
Since $[F_j]_{n,\bar{\epsilon}} = \text{o}(1)$, the limiting operator  $\mathscr{F}_\infty$   has constant coefficients. From uniform convergence,   $\mathscr{F}_\infty$ too satisfies  \eqref{H1} and \eqref{H2}. Also, by the same reasoning employed in \cite{C1}, Lemma 13, we deduce 
\begin{equation}\label{proof comp eq4}
	\mathscr{F}_\infty(D^2u_\infty) = 0
\end{equation}
in the viscosity sense. Furthermore, by Caffarelli's $C^{2,\bar{\epsilon}}$ regularity estimate, see \cite{C1} Theorem 3, and the fact that 
$|D^2u_j(0)| = \text{o}(1)$, we conclude 
\begin{equation}\label{proof comp eq5}
	D^2 u_\infty(0) = 0.
\end{equation}
We have proven that the limiting function $u_\infty$ lies in the functional space $\mathbb{H}$. Hence, we reach a contradiction on \eqref{proof comp eq3} if we take $h = u_\infty$ and $j \gg 1$. The proof of Lemma \ref{comp} is concluded.
\end{proof}

\subsection{Discrete regulariy}

Our next step is to establish a discrete version of the main Theorem. It states that if the oscillations of the data of the equation are universally small, then a step-one version of the {\it continuous} thesis of Thereom \ref{main} holds. More precisely, we have

\begin{lemma} \label{key} Let $u\in C(Q_1)$ be a viscosity solution to $F(X, D^2u) = f(X)$ in $Q_1$ and $|u|\le 1$. Assume $F \colon Q_1 \times \mathrm{Sym}(d) \to \mathbb{R}$ satisfies conditions \eqref{H0}, \eqref{H1}, \eqref{H2} and $f(0) = 0$. Given $0<\beta <\alpha_F$, there exists an $\varepsilon > 0$, that depends only on $\beta$ and universal parameters such that if 
\begin{equation}\label{lemma key eq1}
	[F]_{n,\bar{\epsilon}} + [f]_{n,\gamma} + |D^2u(0)|  \le \varepsilon,
\end{equation}
then there exist numbers $0 < \theta \ll 1 \ll C < \infty$, depending only on $\beta$ and universal parameters, and an affine function $\ell_0(X) = a_0 + \vec{b}_0 \cdot X$,  with universally bounded coefficients, 
$$
	|a_0|_\mathbb{R} + |\vec{b}_0|_{\mathbb{R}^d} \le C
$$
such that the following control is granted 
\begin{equation}\label{key eq2}
	\sup\limits_{Q_\theta} | u(X) - \ell_0(X)| \le \theta^{2+\beta}.
\end{equation}
\end{lemma}

\begin{proof} For a small number $\delta > 0$ to be chosen later, let $h$ be the function sponsored by Lemma \ref{comp}, that is within a $\delta$ distance from $u$ in the $L^\infty$ topology. That is, $h$ solves
\begin{equation}\label{proof key eq0}
	\mathscr{F}(D^2h) = 0 \text{ in } Q_{1/2},
\end{equation} 
for some constant coefficient operator $\mathscr{F}$ which satisfies \eqref{H0} and \eqref{H1} and also $D^2h(0) = 0$ is verified. We notice that within the scope of From Lemma \ref{comp},  once we adjust $\delta>0$  by a universal decision, the choice of $\varepsilon>0$ in the statement of this present Lemma shall also be universal. 

\par

It follows from the $C^{2,\alpha_F}$ regularity theory available for the operator $\mathscr{F}$ that
\begin{equation}\label{proof key eq00}
	\ |h(X) - [h(0) + \nabla h(0) \cdot X]| \le \Theta |X|^{2+\alpha_F},
\end{equation} 
where $\Theta >0$ is the universal constant from \eqref{H1}.  In the sequel we elect $\theta$ and $\delta$ as
\begin{eqnarray}	
	\theta &:=& \sqrt[\alpha_F - \beta]{\dfrac{1}{2\Theta}} \label{proof key eq2}\\
	 \delta &:=& \dfrac{1}{2} \left ( \dfrac{1}{2\Theta}\right )^{\frac{2+\beta}{\alpha_F - \beta}}. \label{proof key eq2.1}
\end{eqnarray}
The selections above depend only on $\beta$ and universal parameters. We further take 
$$
	\ell_0(X) := h(0) + \nabla h(0) \cdot X.
$$ 
Finally  we estimate
\begin{equation}\label{proof key eq1}
	\begin{array}{lll}
		\displaystyle\sup\limits_{Q_\theta} |u(X) - [h(0) + \nabla h(0) \cdot X]| &\le& \displaystyle\sup\limits_{Q_\theta} |u(X) - h(X) | \\
		&+& \sup\limits_{Q_\theta}  |h(X) - [h(0) + \nabla h(0) \cdot X]| \\
		&\le & \delta + \Theta \cdot \theta^{2+\alpha_F} \\
		&=& \theta^{2+\beta},
	\end{array}
\end{equation}
and the proof of Lemma \ref{key} is complete.
\end{proof}

\subsection{Iterative flatness improvement} 

We have now gathered all the elements we need to conclude the proof of Theorem \ref{main}. Let us fix a number $\beta \in (0, \alpha_F) \cap (0, \gamma]$. By normalization, expanding variables and translating the equation, if necessary, we can start off the proof by assuming, with no loss of generality, that 
\begin{equation}\label{proof main eq0}
	f(0) = 0, \quad |u|\le 1, \quad [F]_{n,\bar{\epsilon}} + [f]_{n, \gamma} \le \varepsilon,
\end{equation}
where $\varepsilon>0$ is the universal number found in Lemma \ref{key}, for our fixed choice of $\beta$. 
Recall it is part of the key hypothesis of Theorem \ref{main} that
\begin{equation}\label{proof main eq00}
	D^2u(0) = 0.
\end{equation}

Our strategy is to show the existence of a sequence of affine functions $\ell_k = a_k + b_k X$, satisfying
\begin{equation}\label{proof main eq1}
	|a_k - a_{k-1}|_{\mathbb{R}} + \theta^k |\vec{b}_k - \vec{b}_{k-1}|_{\mathbb{R}^d} \le C \cdot \theta^{k(2+\beta)}
\end{equation}
and
\begin{equation}\label{proof main eq2}
	\sup\limits_{B_{\theta^k}} |u(X) - \ell_k(X) |\le \theta^{k(2+\beta)}.
\end{equation}

We establish \eqref{proof main eq1} and \eqref{proof main eq2} by an induction argument. The step $k=1$ is precisely the statement of Lemma \ref{key}. Suppose we have verified the $k$th step of induction. Define the normalized function
\begin{equation}\label{proof main eq3}
	v_k(X) := \frac{u(\theta^kX) - \ell_k(\theta^kX)}{\theta^{k(2+\beta)}},
\end{equation}
the fully nonlinear operator
\begin{equation}\label{proof main eq4}
	F_k(X, M) := \dfrac{1}{\theta^{k\beta}}F(\theta^{k}X, \theta^{k\beta}M),
\end{equation}
and the source function
\begin{equation}\label{proof main eq5}
	f_k(X) := \dfrac{1}{\theta^{k\beta}}f(\theta^{k}X).
\end{equation}
Easily one verifies that
\begin{equation}\label{proof main eq6}
	F_k(X, D^2v) = f_k(X)
\end{equation}
in the viscosity sense and that $v$, $F_k$ and $f_k$ satisfy all the hypotheses from Lemma \ref{key}. Thus, it follows from the thesis of that Lemma the existence of an affine function $\ell_\star(X) = a_\star + \vec{b}_\star \cdot X$, with 
\begin{equation}\label{proof main eq7}
	|a_\star|_\mathbb{R} + |b_\star|_{\mathbb{R}^d}  \le C, 
\end{equation}
such that
\begin{equation}\label{proof main eq8}
	\sup\limits_{Q_{\theta}} |v(X) - \ell_\star(X)| \le \theta^{2+\beta}.
\end{equation}
Finally if we define
\begin{equation}\label{proof main eq9}
	\ell_{k+1}(X) := \ell_k(X) + \theta^{k(2+\beta)} \ell_\star(\theta^{-k}X),
\end{equation}
and translate back estimate \eqref{proof main eq8} to the original function $u$, we obtain the $(k+1)$th step of the induction process. Furthermore, estimate \eqref{proof main eq1} follows readily from \eqref{proof main eq7} and \eqref{proof main eq9}.

It is now a mere of routine to conclude the proof of Theorem \ref{main}, see for instance \cite{T1}.  We omit the details here.  \hfill $\square$

\bigskip

\bigskip

\noindent{\bf Ackwnoledgement.} This work has been partially supported by CNPq-Brazil and Capes-Brazil. 
\bibliographystyle{amsplain, amsalpha}

\begin{thebibliography}{60}


\bibitem{C1} Caffarelli, Luis A. \textit{Interior a priori estimates for solutions of fully nonlinear equations.} Ann.
of Math. (2) \textbf{130} (1989), no. 1, 189--213.

\bibitem{CC}  Caffarelli, Luis A.; Cabr\'e, Xavier, {\it Fully nonlinear elliptic equations.} American Mathematical Society Colloquium Publications, 43. American Mathematical Society, Providence, RI, 1995. vi+104 pp.
\bibitem{CJK}  Caffarelli, Luis A.; Jerison, David; Kenig, Carlos E. {\it Some new monotonicity theorems with applications to free boundary problems.} Ann. of Math. (2) {\bf 155} (2002), no. 2, 369--404.

\bibitem{CS}  Caffarelli, Luis;  Salsa, Sandro {\it A geometric approach to free boundary problems.} Graduate Studies in Mathematics, 68. American Mathematical Society, Providence, RI, 2005. x+270 pp.

\bibitem{F} Federer, Herbert {\it Geometric measure theory.}  Die Grundlehren der mathematischen Wissenschaften, Band 153 Springer-Verlag New York Inc., New York 1969 xiv+676 pp.

\bibitem{GT} Gilbarg, David; Trudinger, Neil S. {\it Elliptic partial differential equations of second order.} Second edition. Grundlehren der Mathematischen Wissenschaften [Fundamental Principles of Mathematical Sciences], 224. Springer-Verlag, Berlin, 1983. xiii+513 pp.

\bibitem{M} Morrey, C.: {\it Multiple Integrals in Calculus of Variations}. Springer, New York (1966)

\bibitem{S1} Schauder, J.; {\it \"Uber lineare elliptische Differentialgleichungen zweiter Ordnung.} (German) Math. Z. {\bf 38} (1934), no. 1, 257--282. 

\bibitem{Simon}  Simon, Leon {\it Schauder estimates by scaling.}  Calc. Var. Partial Differential Equations {\bf 5} (1997), no. 5, 391--407.

\bibitem{Sh}  Shahgholian, Henrik {\it $C^{1,1}$ regularity in semilinear elliptic problems.} Comm. Pure Appl. Math. {\bf 56} (2003), no. 2, 278--281.


\bibitem{T1} Teixeira, Eduardo V. {\it Universal moduli of continuity for solutions to fully nonlinear elliptic equations.} To appear in Arch. Rational Mech. Anal.





\end{thebibliography}

\bigskip
\bigskip

\noindent \textsc{Eduardo V. Teixeira} \\
\noindent Universidade Federal do Cear\'a \\
\noindent Departamento de Matem\'atica \\
\noindent Campus do Pici - Bloco 914, \\
\noindent Fortaleza, CE - Brazil 60.455-760 \\
 \noindent \texttt{teixeira@mat.ufc.br}

\end{document}